\documentclass[12pt]{article}

\usepackage{amsthm}
\usepackage{amssymb}
\usepackage{amsmath}
\usepackage{comment}
\usepackage{thm-restate}
\usepackage{url} 
\usepackage{hyperref}
\usepackage[noabbrev,capitalise]{cleveref}

\usepackage{color}
\usepackage[normalem]{ulem}

\usepackage[margin=1in]{geometry}

\newcommand{\eps}{\varepsilon}
\newcommand{\mc}[1]{\mathcal{#1}}

\renewcommand{\v}{\textup{\textsf{v}}}
\newcommand{\e}{\textup{\textsf{e}}}
\renewcommand{\d}{\textup{\textsf{d}}}

\theoremstyle{plain}
\newtheorem{thm}{Theorem}[section]
\newtheorem{lem}[thm]{Lemma}
\newtheorem{claim}{Claim}[thm]
\newtheorem{proposition}[thm]{Proposition}
\newtheorem{cor}[thm]{Corollary}

\newtheorem{conj}[thm]{Conjecture}
\newtheorem{que}[thm]{Question}

{\noindent \emph{Proof.} {}{#1}{}}{\hfill
	$\Diamond$\vspace{1em}}

\theoremstyle{plain} 
\newcommand{\thistheoremname}{}
\newtheorem{genericthm}[section]{\thistheoremname}

\theoremstyle{definition}

\newcommand{\ex}{\operatorname{ex}}

\title{Breaking the degeneracy barrier for coloring graphs with no $K_t$ minor}

\author{Sergey Norin\thanks{Department of Mathematics and Statistics, McGill University. Email: {\tt sergey.norin@mcgill.ca}. Supported by an NSERC Discovery grant.}
\and 
 Luke Postle\thanks{Department of Combinatorics and Optimization, University of Waterloo, Waterloo, Ontario, Canada. Email: {\tt lpostle@uwaterloo.ca}. Canada Research Chair in Graph Theory. Partially supported by NSERC under Discovery Grant No. 2019-04304, the Ontario Early Researcher Awards program and the Canada Research Chairs program.}
\and
Zi-Xia Song\thanks{Department of Mathematics, University of Central Florida, Orlando, FL 32816, USA. Email: {\tt Zixia.Song@ucf.edu}. Supported by the National Science Foundation under Grant No. DMS-1854903.}}

\begin{document}

\maketitle

\begin{center}
	\emph{Dedicated to the memory of Robin Thomas}
\end{center}

\begin{abstract} 
	In 1943, Hadwiger conjectured that every graph with no $K_t$ minor is $(t-1)$-colorable for every $t\ge 1$.
	In the 1980s, Kostochka and Thomason independently proved that every graph with no $K_t$ minor has average degree $O(t\sqrt{\log t})$ and hence is $O(t\sqrt{\log t})$-colorable.  We show that every graph with no $K_t$ minor is $O(t(\log t)^{\beta})$-colorable for every $\beta > 1/4$, making the first improvement on the order of magnitude of the Kostochka-Thomason bound.
\end{abstract}

\section{Introduction}

All graphs in this paper are finite and simple. Given graphs $H$ and $G$, we say that $G$ has \emph{an $H$ minor} if a graph isomorphic to $H$ can be obtained from a subgraph of $G$ by contracting edges. We denote the complete graph on $t$ vertices by $K_t$.

In 1943 Hadwiger made the following famous conjecture.

\begin{conj}[Hadwiger's conjecture~\cite{Had43}]\label{Hadwiger} For every integer $t \geq 1$, every graph with no $K_{t}$ minor is $(t-1)$-colorable. 
\end{conj}

Hadwiger's conjecture is widely considered among the most important problems in graph theory and has motivated numerous developments in graph coloring and graph minor theory. We briefly overview major progress towards the conjecture below, and refer the reader to a recent survey by Seymour~\cite{Sey16Survey} for further  background.

Hadwiger~\cite{Had43} and Dirac~\cite{Dirac52} independently showed that \cref{Hadwiger} holds for $t \leq 4$. Wagner~\cite{Wagner37} proved that for $t=5$ the conjecture is equivalent to the Four Color Theorem, which was subsequently proved by Appel and Haken~\cite{AH77,AHK77} using extensive computer assistance.  

Robertson, Seymour and  Thomas~\cite{RST93} went one step further and proved Hadwiger's conjecture for $t=6$, also by reducing it to the Four Color Theorem. Settling the conjecture for $t \geq 7$ appears to be extremely challenging, perhaps in part due to the absence of a transparent proof of  the Four Color Theorem.

Another notable challenging case of Hadwiger's conjecture is the case of graphs with no independent set of size three.
If $G$ is such a graph on $n$ vertices then properly coloring $G$ requires at least $n/2$ colors, and so  Hadwiger's conjecture implies that $G$ has a $K_{\lceil n/2 \rceil}$ minor. This is still open. In fact, as mentioned in~\cite{Sey16Survey}, it is not known whether there exists any $c>1/3$ such that every graph $G$ as above has a $K_{t}$ minor for some $t \geq cn$. 

The following natural weakening of Hadwiger's conjecture, which has been considered by several researchers, sidesteps the above challenges.  

\begin{conj}[Linear Hadwiger's conjecture~\cite{ReeSey98,Kaw07, KawMoh06}]\label{c:LinHadwiger} There exists $C>0$ such that for every integer $t \geq 1$, every graph with no $K_{t}$ minor is $Ct$-colorable. 
\end{conj}

In this paper we take a step towards \cref{c:LinHadwiger} by  improving for large $t$ the upper bound on the number of colors needed to color graphs with no $K_{t}$ minor. Prior to our work, the best bound was $O(t\sqrt{\log{t}})$, which was obtained independently by Kostochka~\cite{Kostochka82,Kostochka84} and Thomason~\cite{Thomason84} in the 1980s. The only improvement~\cite{Thomason01,Wood13,KelPos19} since then has been in the constant factor. 

The results of \cite{Kostochka82,Kostochka84,Thomason84} bound the ``degeneracy" of graphs with no $K_t$ minor.
Recall that a graph $G$ is \emph{$d$-degenerate} if every non-null subgraph of $G$ contains a vertex of degree at most $d$. A standard inductive argument shows that every $d$-degenerate graph is $(d+1)$-colorable. Thus the following bound on the 
degeneracy of graphs with no $K_t$ minor gives a corresponding bound on their chromatic number. 

\begin{thm}[\cite{Kostochka82,Kostochka84,Thomason84}]\label{t:KT} Every graph with no $K_t$ minor is $O(t\sqrt{\log{t}})$-degenerate.
\end{thm}

Kostochka~\cite{Kostochka82,Kostochka84} and de la Vega~\cite{Vega83} have shown that there exist graphs with no $K_t$ minor and minimum degree $\Omega(t\sqrt{\log{t}})$. Thus the bound in \cref{t:KT} is tight, and it is natural to consider the possibility that coloring graphs with no $K_t$ minor requires $\Omega(t\sqrt{\log{t}})$ colors. In fact, Reed and Seymour~\cite{ReeSey98} refer to this assertion as ``a commonly expressed counter-conjecture" to Conjecture~\ref{Hadwiger}. 

We disprove the above ``counter-conjecture" by proving the following main result.

\begin{thm}\label{t:main} For every $\beta > \frac 1 4$,
every graph with no $K_t$ minor is $O(t (\log t)^{\beta})$-colorable.
\end{thm}

The proof of \cref{t:main} occupies most of the paper. In Section~\ref{s:outline}, we outline the proof, and derive  \cref{t:main} from our two main technical results:
\begin{itemize}
	\item \cref{t:newforced}, which shows that any sufficiently dense graph with no $K_t$ minor contains a relatively small subgraph, which is still dense, and
	\item \cref{t:minorfrompieces}, which  adapts an argument of Thomason~\cite{Thomason01} to show that every sufficiently well-connected graph containing a large number of vertex-disjoint dense subgraphs has a $K_t$ minor.
\end{itemize}	
 We prove \cref{t:minorfrompieces} in Section~\ref{s:build}. We prove \cref{t:newforced} in Section~\ref{s:force} using a density increment argument. In \cref{s:turan} we give an application of \cref{t:newforced} beyond the proof of \cref{t:main}, extending results of K\"{u}hn and Osthus~\cite{KO03,KO05}, and Krivelevich and Sudakov~\cite{KriSud09} on the density of minors in graphs with a forbidden complete bipartite subgraph or an even cycle to general bipartite graphs. In Section~\ref{s:remarks} we conclude the paper with a few  remarks.
 
\subsubsection*{Notation}

We use largely standard graph-theoretical notation. We denote by $\v(G)$ and $\e(G)$ the number of  vertices and edges of a graph $G$, respectively, and denote by $\d(G)=\e(G)/\v(G)$ the \emph{density} of a non-null graph $G$. We use $\chi(G)$ to denote the chromatic number of $G$, and $\kappa(G)$ to denote the (vertex) connectivity of $G$. 
The degree of a vertex $v$ in a graph $G$ is denoted by $\deg_G(v)$ or simply by $\deg(v)$ if there is no danger of confusion.
We denote by $G[X]$ the subgraph of $G$ induced by a set $X \subseteq V(G)$. 

For a positive integer $n$, let $[n]$ denote $\{1,2,\ldots,n\}$. The logarithms in the paper are natural unless specified otherwise.

We say that vertex-disjoint subgraphs $H$ and $H'$ of a graph $G$ are \emph{adjacent} if there exists an edge of $G$ with one end in $V(H)$ and the other in $V(H')$, and $H$ and $H'$ are \emph{non-adjacent}, otherwise.

A collection $\mc{X} = \{X_1,X_2,\ldots,X_h\}$ of pairwise disjoint subsets of $V(G)$ is a \emph{model of a graph $H$ in a graph $G$} if $G[X_i]$ is connected for every $i \in [h]$, and there exists a bijection $\phi: V(H) \to [h]$ such that $G[X_{\phi(u)}]$ and $G[X_{\phi(v)}]$ are adjacent for every $uv \in E(H)$. It is well-known and not hard to see that $G$ has an $H$ minor if and only if there exists a model of $H$ in $G$. We say that a model $\mc{X}$ as above is \emph{rooted at $S$} for $S \subseteq V(G)$ if $|S|=h$  and  $|X_i \cap S|=1$ for every $i \in [h]$.

\section{Outline of the proof}\label{s:outline}

\subsubsection*{Small graphs.}

The proof of \cref{t:main} uses different methods depending on the magnitude of $\v(G)$.
In the case when $\v(G)$ is small our argument is based on the following classical bound due to Duchet and Meyniel~\cite{DucMey82} on the independence number of graphs with no $K_t$ minor.

\begin{thm}[\cite{DucMey82}]\label{t:DucMey} For every $t \geq 2$,
	every graph $G$ with no $K_t$ minor has an independent set of size at least $\frac{\v(G)}{2(t-1)}$. 
\end{thm}
 
\cref{t:DucMey} implies that every graph with no $K_t$ minor contains a $t$-colorable subgraph on a constant proportion of vertices. Woodall~\cite{Woodall87} proved the  following stronger result, which as observed by Seymour~\cite{Sey16Survey} also follows from the proof of \cref{t:DucMey} in \cite{DucMey82}.
 
\begin{thm}[\cite{Woodall87}]\label{t:big}
 	Let $G$ be a graph with no $K_t$ minor. Then there exists $X \subseteq V(G)$ with $|X| \geq \frac{\v(G)}{2}$ such that $\chi(G[X]) \leq t-1$. 
\end{thm}

\cref{t:big} straightforwardly implies the following bound on the chromatic number of graphs with no $K_t$ minor.

\begin{cor}\label{c:big}
 	Let $G$ be a graph with no $K_t$ minor. Then $$\chi(G) \leq \left(\log_{2}\left(\frac{\v(G)}{t}\right)+2\right)t.$$
\end{cor}
 
\begin{proof}
 	By \cref{t:big} for every integer $s \geq 0$ there exist disjoint subsets $X_1,X_2,$ $\ldots,X_s$  $\subseteq V(G)$ such that $|V(G) - \cup_{i=1}^{s}X_i| \leq \v(G)/2^s$ and $\chi(G[X_i]) \leq t$ for every $i\in [s]$.
 	Let $s=\lceil \log_{2}(\v(G)/t) \rceil$. Then  $\v(G)/2^s \leq t$, and so $$\chi(G\setminus\cup_{i=1}^{s}X_i) \leq \v(G\setminus\cup_{i=1}^{s}X_i) \leq t.$$ It follows that $\chi(G) \leq t + \sum_{i=1}^s\chi(G[X_i]) \leq (s+1)t$, implying the corollary.
\end{proof}

By \cref{c:big} we may assume that $\v(G)$ is large. Given a graph $G$ with $\chi(G) = \Omega(t (\log t)^{\beta})$, where $\beta$ is as in \cref{t:main}, we find a $K_t$ minor in $G$ by adapting the following strategy employed by Thomason~\cite{Thomason01}: We first construct a large collection of pairwise vertex-disjoint dense subgraphs $H_1,\ldots,H_r$ of $G$,  and then find a model of a smaller complete graph in each $H_i$, and link these models together to build a model of $K_t$.

\subsubsection*{Density increment.}

The  next theorem is the central element of our proof. We precede its statement with a brief motivation. 

By \cref{t:KT} there exists $D=O(t\sqrt{\log t})$ such that every graph $G$ with  density $\d(G) \geq D$ has a $K_t$ minor. For a graph $G$ with smaller density one might still hope to guarantee a $K_t$ minor by finding a minor $H$ of $G$ with $\d(H) \geq D$. 
Thus we are interested for given $d,D$  in properties of  graphs $G$ of density $\d(G)=d$ and no minor of density $D$, which are the \emph{obstructions} to this approach.

It is possible that such a graph $G$ simply does not have enough edges. As every graph of density $D$ has  at least $D^2$ edges, if $G$ has a minor of density $D$ we must have $ D^2 \leq \e(G) = d\cdot \v(G)$. It follows that all the graphs $G$ with $\v(G) < D^2/d$ are among the obstructions. One can obtain further obstructions by taking disjoint union of such graphs, and, more generally, by gluing smaller obstructions along small sets in a ``tree-like fashion''. Note that all graphs obtained in this way contain a subgraph with at most $D^2/d$ vertices and density close to $d$.

Our result shows that all the obstructions have a similar property, i.e. they contain a subgraph of density $d/K$ on at most $KD^2/d$ vertices, where the error factor $K$ is subpolynomial in the density gap $D/d$.

\begin{thm}\label{t:newforced} For every $\delta > 0$ there exists $C=C_{\ref{t:newforced}}(\delta) > 0$ such that for every $D > 0$ the following holds. Let $G$ be a graph with $\d(G) \ge C$, and let $s=D/\d(G)$.  Then $G$ contains at least one of the following: \begin{description}
		\item[(i)] a minor $J$ with $\d(J) \geq D$, or
		\item[(ii)] a subgraph $H$ with $\v(H) \leq Cs^{\delta} D^2/\d(G)$ and $\d(H) \geq s^{-\delta}\d(G)/C$.
	\end{description}  
\end{thm}

\cref{t:newforced} has applications beyond the proof of \cref{t:main}, one of which is given in \cref{s:turan}. As for the proof of \cref{t:main}, \cref{t:newforced} allows us to extract subgraphs one by one to construct the collection $H_1,\ldots,H_r$ mentioned at the end of the last subsection as otherwise we can partition our graph into two subgraphs, a small subgraph and a sparse subgraph, both of which are colorable with few colors: the small subgraph by \cref{c:big} and the sparse subgraph by the standard degeneracy argument.

\cref{t:newforced} is derived from the following  ``density increment'' result. Its statement requires one additional definition. We say that a graph $H$ is a \emph{$k$-bounded minor} of a graph $G$ if there exists a model $\mc{X}$ of $H$ in $G$ such that $|X| \leq k$ for every $X \in \mc{X}$. That is, $H$ can be obtained from a subgraph of $G$ by contracting connected subgraphs on at most $k$ vertices.

\begin{restatable}{thm}{Dense}\label{DenseSubgraph2}
Let $k\ge \ell \ge 2$ be integers. Let  $\varepsilon \in \left(0,\frac{1}{6k}\right)$ and let $G$ be a graph with $d=\d(G) \ge 1/\varepsilon$. Then $G$ contains at least one of the following:
	\begin{description}
		\item[(i)] a subgraph $H$ with $\v(H)\le 3k^3d$ and $\e(H)\ge \varepsilon^2d^2/2$, or
		\item[(ii)] an $(\ell+1)$-bounded minor $G'$ with $\d(G') \ge \frac{\ell}{2} (1-6k\varepsilon) d$, or
		\item[(iii)] a $k$-bounded minor $G'$ with $\d(G') \ge \frac{k}{8\ell} (1-2k\varepsilon) d$.
	\end{description}  
\end{restatable}

The proof of \cref{DenseSubgraph2} is the most challenging part of our argument and occupies Section~\ref{s:force}. Meanwhile, let us derive \cref{t:newforced} from \cref{DenseSubgraph2}.

\begin{proof}[Proof of \cref{t:newforced}]
	For $\delta > 0$, let  integers $k \ge \ell \ge 2$, $\eps > 0$ be chosen so that
	\begin{align}
	\left(1-6k\varepsilon\right)\frac{\ell}{2} &\geq (\ell+1)^{1/(\delta+1)}, \qquad \mathrm{and} \label{e:l}	\\	\left(1-2k\varepsilon \right)  \frac{k}{8\ell} &\geq k^{1/(\delta+1)}. \label{e:k} 
	\end{align}
	It is easy to see that such a choice is possible. We show that $C=C_{\ref{t:newforced}}(\delta)=6k^3/\eps^2$  satisfies the theorem.
	
\begin{claim}\label{c:dense3}
	Every graph $G$ with $\d(G) \ge C$ contains at least one of the following:
	\begin{enumerate}
		\item[(a)] a subgraph $H$ with $\v(H)\le Cd$ and $\d(H) \geq \d(G)/C$, or
		\item[(b)] an $r$-bounded minor $G'$ with $\d(G') \ge r^{1/(\delta+1)} \d(G)$ for some $r > 1$.
	\end{enumerate}
\end{claim}	
\begin{proof} We apply \cref{DenseSubgraph2} to $G$. If \cref{DenseSubgraph2}(i) holds then (a) holds by the choice of $C$. If \cref{DenseSubgraph2}(ii) or (iii) holds then (b) holds by (\ref{e:l}) and (\ref{e:k}), respectively.	
\end{proof}	

Suppose now for a contradiction, that there exists a graph $G$ with $\d(G) \geq C$ that does not satisfy the conclusion of \cref{t:newforced}, while every proper minor $H$ with $\d(H) \geq C$ of $G$ satisfies it. Thus  $G$ has no minor $J$ with $\d(J) \geq D$. In particular $\d(G) < D$, and $s=D/\d(G) > 1$. If there exists a subgraph $H$ of $G$ as in \cref{c:dense3}(a), then  \cref{t:newforced}(ii) holds, contrary to the choice of $G$.

Thus by \cref{c:dense3}(b), $G$ has an $r$-bounded minor $G'$ with $\d(G') \ge r^{1/(\delta+1)} \d(G)$ for some $r > 1$. By the choice of $G$,  $G'$ has a subgraph $H'$ with $$\v(H') \leq (s')^{1+\delta} CD \qquad \mathrm{and} \qquad \d(H') \geq \frac{(s')^{-\delta}\d(G')}{C},$$
where $s' = D/\d(G') \leq sr^{-1/(\delta+1)}$. Then $H'$ is an $r$-bounded minor of $G$, corresponding to a subgraph $H$ of $G$ with $\v(H) \leq r\v(H')$ and $\d(H) \geq \d(H')/r$. Thus 
\begin{align*}\v(H) &\leq r(s')^{1+\delta} CD \leq  s^{1+\delta} CD, \qquad \mathrm{and}\\
	\d(H) &\geq \frac{(s')^{-\delta}\d(G')}{Cr} \geq  \left(sr^{-1/(\delta+1)}\right)^{-\delta}r^{1/(\delta+1)}r^{-1} \d(G)/C = s^{-\delta}\d(G)/C,
\end{align*}
and so \cref{t:newforced}(ii) holds, contradicting the choice of $G$.
\end{proof}	

\subsubsection*{Building a $K_t$ minor.}

Once appropriate $H_1,\ldots,H_r$ are found via the repeated use of Theorem~\ref{t:newforced}, the following theorem ensures the existence of a $K_t$ minor.

\begin{restatable}{thm}{minorfrompieces}\label{t:minorfrompieces} There exists $C=C_{\ref{t:minorfrompieces}} >1$ satisfying the following. 
	Let $G$ be a graph with $\kappa(G) \geq  Ct(\log t)^{1/4}$, and let $r \geq \sqrt{\log t}/2$ be an integer. If 
	there exist pairwise vertex-disjoint subgraphs $H_1,H_2,\ldots,H_{r}$ of $G$ such that $\d(H_i) \geq Ct(\log t)^{1/4}$ for every $i \in [r]$  then $G$ has a $K_t$ minor. 
\end{restatable}

\cref{t:minorfrompieces} is proved in \cref{s:build}. To apply it in the proof of Theorem~\ref{t:main}, we need a bound on the connectivity of $G$. We say that a graph $G$ is \emph{contraction-critical} if  $\chi(H) < \chi(G)$ for every proper minor $H$ of $G$.
Clearly, a minimum counterexample to Theorem~\ref{t:main}   is contraction-critical. This allows us to use the connectivity bound established by Kawarabayshi~\cite{Kaw07}.

\begin{thm}[\cite{Kaw07}]\label{t:minconnectivity}
	Let $G$ be a contraction-critical graph with $\chi(G) \geq k$. Then $\kappa(G) \geq 2k/27$.
\end{thm}

In the remainder of this section we deduce \cref{t:main} from Theorems~\ref{t:newforced},~\ref{t:minorfrompieces},~\ref{t:minconnectivity}, Corollary~\ref{c:big} and the following explicit form of \cref{t:KT} from~\cite{Kostochka82}.

\begin{thm}[\cite{Kostochka82}]\label{t:density}
 	Let $t \geq 2$ be an integer. Then every graph $G$  with $\d(G) \geq 3.2 t \sqrt{\log t}$ has a $K_t$ minor.
\end{thm}

\begin{proof}[Proof of \cref{t:main}]  It suffices to show that for every $\delta>0$  there exists $t_0=t_0(\delta)$ such that for all positive integers $t \geq t_0$,  every graph $G$ with no $K_t$ minor satisfies $$\chi(G)< t(\log t)^{\frac14   + \delta}.$$
	We assume without loss of generality that $\delta < 1/4$.
	
Let $C_1 = C_{\ref{t:newforced}}(\delta)$, and let  $C_2=C_{\ref{t:minorfrompieces}}$. We choose $t_0 \gg C_1,C_2,1/\delta$ implicitly to satisfy the inequalities appearing throughout the proof.
	
Let $t \geq t_0$ be an integer and let $k = t(\log t)^{\frac14   + \delta}$. Suppose for a contradiction that there exists  a graph $G$ with no $K_t$  minor such that $\chi(G) \geq k$. We assume without loss of generality that $G$ is contraction-critical. Thus $\kappa(G) \geq 2k/27$ by \cref{t:minconnectivity}. In particular,  $\kappa(G) \geq  C_2t(\log t)^{1/4}$ for large enough $t$. 

	Choose a maximal collection $H_1,H_2,\ldots,H_{r}$ of pairwise vertex-disjoint subgraphs of $G$ such that $\d(H_i) \geq C_2t(\log t)^{1/4}$ and $\v(H_i) \leq t(\log t)^{3/4}$. 
 By the choice of $G$ and \cref{t:minorfrompieces} we have $r < \sqrt{\log t}/2$. Let $X = \cup_{i \in [r]}V(H_i)$. Then $|X| < t (\log t)^{5/4}$.
	By \cref{c:big} for sufficiently large $t$ we have	$$\chi(H[X]) \leq 2t \log\log t < k/2-1.$$ Thus $\chi(G \setminus X) \geq k/2+1$. 
	
	Let $G'$ be a minimal subgraph of $G \setminus X$ such that $\chi(G') \geq k/2+1$. Then every vertex of $G'$ has degree at least $k/2$, and so $\d(G') \geq k/4$. Let $D= 3.2 t \sqrt{\log t}$. We apply \cref{t:newforced} to $D$ and $G'$. If $G'$ has a minor $J$ with $\d(J) \geq D$, then $G'$ has a $K_{t}$ minor by \cref{t:density}, contradicting the choice of $G$.
	Thus there exists a subgraph $H$ of $G'$ such that  $\v(H) \leq s^{1+\delta} C_1D$ and $\d(H) \geq s^{-\delta}\d(G')/C_1$, where $s=D/\d(G') \leq 13(\log{t})^{1/4-\delta}$.  It is easy to check that for large enough $t$ the above conditions imply $\d(H) \geq  C_2t(\log t)^{1/4}$ and $\v(H) \leq t(\log t)^{3/4}$. Thus the collection  $\{H_1,H_2,\ldots,H_{r},H\}$ contradicts the maximality of $\{H_1,H_2,\ldots,H_{r}\}$.
\end{proof}

\section{Building a $K_t$ minor}\label{s:build}

In this section we prove \cref{t:minorfrompieces}. Our proof uses several additional tools from the literature.

First, we will need each subgraph $H_i$ in the statement of \cref{t:minorfrompieces} to be not only dense, but highly-connected. This is not hard to guarantee  using
a classical result of Mader~\cite{Mader72} which ensures that every dense graph contains a highly-connected subgraph.

\begin{lem}[\cite{Mader72}]\label{l:connect}
	Every graph $G$ contains a subgraph $G'$ such that $\kappa(G') \geq \d(G)/2$.
\end{lem}

The technical part of the proof of \cref{t:minorfrompieces} involves linking the models we construct in each $H_i$. To accomplish this we employ a toolkit introduced by Bollob\'as and Thomason~\cite{BolTho96} for finding rooted models in highly connected graphs.

\begin{lem}[\cite{BolTho96}]\label{l:denseminor1}
	Let $G$ be a graph with $d= \d(G) \geq 3$. Then $G$ has a  minor $H$  such that $v(H) \leq d+2$ and $2\delta(H) \geq\v(H) + 0.3d -2$.
\end{lem}

\begin{lem}[\cite{BolTho96}]\label{l:denseminor2}
	Let $n \geq 0, k \geq 2$ and $h \geq n+ 3k/2$ be integers. Let $G$ be a graph with $\kappa(G) \geq k$ containing vertex-disjoint non-empty connected subgraphs $C_1,\ldots,C_h$ such that each of them is non-adjacent to at most $n$ others. Let $S=\{s_1,\ldots,s_k\} \subseteq V(G)$. Then $G$ contains vertex-disjoint non-empty connected subgraphs $D_1,\ldots, D_m$ where $m = h -\lfloor k/2\rfloor$, such that $s_i \in V(D_i)$ for each $i \in [k]$ and every element of  $\{D_1,\ldots, D_m\}$ is non-adjacent to at most $n$ subgraphs among $D_{k+1},\ldots, D_m$.
\end{lem}

It is worth noting that \cref{l:denseminor2} corresponds to \cite[Lemma 2]{BolTho96}, where the last condition is only stated for subgraphs $D_1,\ldots,D_k$, but the family $D_1,\ldots, D_m$ constructed in the proof has the stronger condition claimed in \cref{l:denseminor2}.

In addition to the above lemmas, we also use one of the main results of~\cite{BolTho96}.

\begin{thm}[\cite{BolTho96}]\label{t:knitted}
 There exists $C=C_{\ref{t:knitted}}> 0$ satisfying the following. Let $s$ be a positive integer, let $G$ be a graph with $\kappa(G) \geq Cs$, and let $S_1,S_2,\ldots,S_k$ be non-empty disjoint subsets of $V(G)$ such that $\sum_{i=1}^{k}|S_i| \leq s$. Then there exist vertex-disjoint connected subgraphs $C_1,\ldots,C_k$ of $G$ such that $S_i \subseteq V(C_i)$ for every $i \in [k]$.
\end{thm}

The value of $C_{\ref{t:knitted}}$ is not explicitly given in~\cite{BolTho96}, but it is not hard to see that $C_{\ref{t:knitted}}=22$ suffices. Thomas and Wollan~\cite{ThoWol05} improve the bounds from~\cite{BolTho96}, and the results of~\cite{ThoWol05} directly imply that  $C_{\ref{t:knitted}}=10$ satisfies \cref{t:knitted}.  The exact value of $C_{\ref{t:knitted}}$  does not substantially affect our bounds.    

The next lemma  is used to construct the pieces of our model of $K_t$. Let $l$ be a positive integer. Given a collection  $\mc{S} = \{(s_i,t_i)\}_{i \in [l]}$ of pairs of vertices of a graph $G$ (where $s_i$ and $t_i$ are possibly the same) \emph{an $\mc{S}$-linkage $\mc{P}$} is a collection of vertex-disjoint paths $\{P_1,\ldots,P_l\}$ in $G$ such that $P_i$ has ends $s_i$ and $t_i$ for every $i \in [l]$.   

\begin{lem}\label{l:rooted} There exists $C=C_{\ref{l:rooted}} >0$ satisfying the following. 
	Let $G$ be a graph,  let $l \geq s \geq 2$ be positive integers.  Let $s_1,\ldots, s_l,$ $t_1,\ldots,t_l,$  $r_1,\ldots, r_s \in V(G)$ be distinct, except possibly $s_i=t_i$ for some $i \in [l]$.
	If $$\kappa(G) \geq C \cdot\max\{l, s\sqrt{\log s}\},$$	
	then there exists a $K_{s}$ model $\mc{M}$ in $G$ rooted at $\{r_1,\ldots,r_s\}$ and an $\{(s_i,t_i)\}_{i \in [l]}$-linkage $\mc{P}$ in $G$ such that $\mc{M}$ and $\mc{P}$ are vertex-disjoint. 
\end{lem}

\begin{proof} Our choice of $C$ will be implicit, i.e.~we assume that it is chosen to satisfy the inequalities appearing throughout the proof.
	
Let $d=\d(G) \geq \kappa(G)/2 \geq Cl/2$. By~\cref{l:denseminor1} there exists a model $\mc{H}$ of a graph $H$ in $G$ such that $\v(H) \leq d+2$, and every vertex in $H$ has at most $\v(H)/2-d/10$ non-neighbors. 

Let $h=\v(H)$, $n=h/2-d/10$, and let $k = 2l+s \leq 3l \leq d/150$. (The last inequality assumes $C \geq 900$.) Then $h \geq n+3k/2$. By~\cref{l:denseminor2} applied to the elements of $\mc{H}$ and $S=\{s_1,\ldots, s_l,t_1,\ldots,t_l, r_1,\ldots, r_s\}$, there exist a collection $\mc{D}= \{D_1,\ldots, D_m\}$ of vertex-disjoint non-empty connected subgraphs $D_1,\ldots, D_m$ of $G$ where $m =  h - \lfloor k/2 \rfloor$, such that $D_1,\ldots,D_{k}$ each contain exactly one vertex from $S$, and every element of $\mc{D}$ is non-adjacent to at most $n$ subgraphs in $\mc{D}':= \{D_{k+1},\ldots, D_m\}$. We may assume without loss of generality that $r_i \in V(D_i)$ for every $i \in [s]$.
As $|\mc{D}'| \geq h-3k/2$, every two elements of $\mc{D}$ have at least $$|\mc{D}'|-2n-2 \geq d/5 - 3k - 2 \geq d/6$$ common neighbors in $\mc{D'}$. 

Let $\mc{D}'' \subseteq \mc{D}'$ be chosen by selecting each element of $\mc{D}'$ independently at random with probability $1/2$. Then by the Chernoff bound the probability that a given pair of elements of $\mc{D}$ have fewer than $d/24$ common neighbors in $\mc{D}''$ is at most $e^{-d/100}$. For sufficiently large $C$ we have $(d+2)^2e^{-d/100} < 1/2$, and thus by linearity of expectation there exists $\mc{D}'' \subseteq \mc{D}'$ such that $|\mc{D}''| \leq h/2$ and every pair of elements of $\mc{D}$ has at least $d/24 \geq k$ common neighbors in $\mc{D}''$. Let $\mc{M}'=\mc{D}'-\mc{D}''$, then $|\mc{M}'| \geq h/2-3k/2$. Note that every element of $\mc{M}'$ is non-adjacent to at most $n$ other elements of $\mc{M}'$,  and hence is adjacent to at least $$|\mc{M}'|-n \geq \left(h-\frac{3k}{2}\right) - \left(h-\frac{d}{5}\right) -2 = \frac{d}{10} -\frac{3k}{2} -2 \geq \frac{d}{12} \geq \frac{C}{24} s\sqrt{\log s}$$ other elements of $\mc{M}'$. By \cref{t:density}, there exists a model $\mc{M}''=\{M'_1,\ldots,M'_s\}$ of $K_s$ in $G$ such that each element of $\mc{M}''$ is a union of vertex sets of elements of $\mc{M}'$. By the choice of $\mc{D}''$, there exists $\{D'_1,\ldots,D'_s\} \subseteq \mc{D}''$ such that  $D'_i$ is adjacent to $D_i$ and $G[M'_i]$ for every $i \in [s]$. Let $\mc{M}=\{M'_i \cup V(D_i) \cup V(D'_i)\}_{i \in [s]}$. Then $\mc{M}$ is a model  of $K_s$ in $G$, rooted at $\{r_1,\ldots,r_s\}$. Similarly, using  $l$  elements of $\mc{D}'' - \{D'_1,\ldots,D'_s\}$, we find an $\{(s_i,t_i)\}_{i \in [l]}$-linkage $\mc{P}$ in $G$, such that $\mc{P}$ is vertex-disjoint from $\mc{M}$, as desired. 			
\end{proof}	

We are now ready to prove  \cref{t:minorfrompieces}, which we restate for convenience.\minorfrompieces*

\begin{proof} Again we will choose $C=C_{\ref{t:minorfrompieces}}$ implicitly, sufficiently large with respect to $C_{\ref{t:knitted}}$ and $C_{\ref{l:rooted}}$. By \cref{l:connect}, replacing each $H_i$ by a subgraph as necessary, we may assume that $$\kappa(H_i) \geq \frac C 2 t(\log t)^{1/4},$$ instead of   $\d(H_i) \geq Ct(\log t)^{1/4}$. Let $y = \lfloor (\log t)^{1/4} \rfloor$ and  $x = \lceil t/ y \rceil$. Then $xy \geq t$ and it suffices to show that $G$ has a $K_{xy}$ minor. We reindex the graphs $H_1,\ldots,H_{\binom{y}{2}+1}$ to $H_0$ and $\{H_{\{i,j\}}\}_{\{i,j\} \subseteq [y]}$. By choosing $C$ appropriately large, we may assume that $\kappa(G) \geq xy(y-1)$. Then it follows from~\cref{t:knitted} that there exist vertex-disjoint linkages $\mc{Q}_{(i,j)}$ for all $i,j \in [y]$ with $i \neq j$, such that each $\mc{Q}_{(i,j)}$ consists of $x$ paths $Q^{1}_{(i,j)}, \ldots, Q^{x}_{(i,j)} $ each starting in $V(H_{\{i,j\}})$, ending in $V(H_0)$ and otherwise disjoint from $V(H_{\{i,j\}}) \cup V(H_0)$. Let $\mc{Q} = \cup_{i,j \in [y], i \neq j} \mc{Q}_{(i,j)}$. 
	
	We now apply \cref{l:rooted} consecutively to each of the subgraphs $H=H_{\{i,j\}}$ with $s=2x$, and $l \leq xy(y-1)-2x$ equal to the number of paths in $\mc{Q} - \mc{Q}_{(i,j)} - \mc{Q}_{(j,i)}$   which intersect $H$. The vertices $\{(s_i,t_i)\}_{i \in [l]}$ are then chosen to be the first and last vertex of these paths in $H$, while the vertices $r_1,r_2,\ldots,r_{s}$ are the ends of the paths  $\mc{Q}_{(i,j)} \cup \mc{Q}_{(j,i)} $ in $H$. By using  the linkage $\mc{P}$ given by  \cref{l:rooted} to reroute the paths in $\mc{Q}  - \mc{Q}_{(i,j)} - \mc{Q}_{(j,i)}$ within $H$, we may assume that $H$ contains a $K_{2x}$ model $\mc{M}_{\{i,j\}}$ rooted at $\{r_1,r_2,\ldots,r_{s}\} \subseteq V(\mc{Q}_{(i,j)}) \cup V(\mc{Q}_{(j,i)} )$, which is  otherwise disjoint from $V(\mc{Q})$.
 		
 	Finally we need to join the ends of paths in $\mc{Q}$ in $H_0$.	
	By \cref{t:knitted} there exist vertex-disjoint connected subgraphs $\{C^{a}_{i}\}^{a \in [x]}_{i\in [y]}$ of $H_0$ such that $V(C^{a}_{i})$	contains the ends of paths $Q^{a}_{(i,j)}$ for all $j \in [y]-\{i\}$, and is otherwise disjoint from $V(\mc{Q})$. These $xy$ connected subgraphs together with the paths of $\mc{Q}$ ending in them, and the elements of the  $K_{2x}$ models containing the second ends of these paths now form the elements of a $K_{xy}$ model in $G$, as desired.
\end{proof}

\section{Finding a small dense subgraph}\label{s:force}

In this section we prove \cref{DenseSubgraph2}. The proof is based on two theorems, \cref{kclawDense} about unbalanced bipartite graphs and \cref{DenseSubgraph3} about general graphs. We prove \cref{kclawDense} in Subsection~\ref{Bip} and \cref{DenseSubgraph3} in Subsection~\ref{Shrub}. Finally in Subsection~\ref{Combined}, we prove \cref{DenseSubgraph2} by combining these two theorems. However, first we will need some preliminaries.

An important concept to the proofs of Theorems~\ref{kclawDense} and~\ref{DenseSubgraph3} is that of a mate, defined as follows. 

Let $G$ be a graph, and let $K,d\ge 1$, $\varepsilon \in (0,1)$ be real. We say that two vertices of $G$ are \emph{$(\varepsilon,d)$-mates} if they have at least $\varepsilon d$ common neighbors. We say that $G$ is \emph{$(K,\varepsilon, d)$-unmated} if  every vertex of degree at most $Kd$ in $G$ has strictly fewer than $\varepsilon d$ $(\varepsilon,d)$-mates.

We need the following useful proposition which shows that if a graph does not contain a small dense subgraph, then every $k$-bounded minor of it is unmated (for the appropriate choice of constants).

\begin{proposition}\label{SmallDenseBounded}
		Let $k,d \ge 1$,  $\varepsilon\in (0,1)$.  If there does not exist a subgraph $H$ of a graph $G$ with $\v(H) \le 3k^3d$ and $\e(H)\ge \varepsilon^2 d^2/2$, then every $k$-bounded minor of $G$ is $(k^2,\varepsilon,d)$-unmated.
\end{proposition}
\begin{proof}
Assume for a contradiction that there exists a $k$-bounded minor $G'$ of $G$ that is not $(k^2,\varepsilon, d)$-unmated. Then there exists $v \in V(G')$ with $\deg_{G'}(v) \leq k^2d$ such that $v$ has at least $\varepsilon d$ $(\varepsilon,d)$-mates in $G'$. Let $v_1, \ldots, v_{\lceil \varepsilon d \rceil}$ be distinct $(\varepsilon,d)$-mates of $v$ in $G'$. Let $H' = G'[N(v) \cup \{v,v_1,\ldots, v_{\lceil \varepsilon d \rceil} \}]$. Then $\v(H') \le 1 + k^2d + \lceil \varepsilon d \rceil \le 3k^2d$ and $\e(H')\ge \varepsilon^2 d^2/2$. Since $H'$ is a $k$-bounded minor of $G$, it corresponds to a subgraph $H$ of $G$ with $\v(H) \le k \cdot \v(H') \leq 3k^3d$ and $\e(H)\ge \e(H')  \ge \varepsilon^2 d^2/2$, a contradiction.
\end{proof}

We also need a few definitions involving forests in a graph as follows. Let $F$ be a forest in a graph $G$. For any real number $k\ge 1$, we say $F$ is \emph{$k$-bounded} if $\v(T)\le k$ for every component $T$ of $F$. For any real numbers $d \ge 1$ and $\varepsilon \in (0,1)$, we say $F$ is \emph{$(\varepsilon,d)$-mate-free (in $G$)}  if no two distinct vertices in any component of $F$ are $(\varepsilon,d)$-mates in $G$. If $G=(A,B)$ is bipartite, then we say $F$ is a \emph{star forest from $B$ to $A$} if every component of $F$ is a star with a center in $B$.

\subsection{Dense minors in unbalanced bipartite graphs}\label{Bip}

In this subsection, we prove the following theorem about unbalanced bipartite graphs using an alternating path argument.


\begin{restatable}{thm}{kClawDense}\label{kclawDense}
	Let $\ell \geq 2$ be an integer, and let $\varepsilon_0 \in (0,\frac{1}{2\ell})$ and  $d_0 \ge 1/\varepsilon_0$ be real. Let $G=(A,B)$ be a bipartite graph such that $|A| > \ell |B|$ and every vertex in $A$ has at least $d_0$ neighbors in $B$. Then $G$ contains at least one of the following:
	\begin{description}
		\item[(i)] a subgraph $H$ with $\v(H) \le 3d_0$ and $\e(H)\ge \varepsilon_0^2 d_0^2/2$, or
		\item[(ii)] an $(\ell+1)$-bounded minor $G'$ with $\d(G') \ge \frac{\ell}{2}(1- 2 \ell \varepsilon_0)d_0$.
	\end{description}
\end{restatable}
\begin{proof}
We may assume without loss of generality that every vertex in $A$ has exactly $d_0$ neighbors in $B$. Now first suppose that $G$ is not $(1,\varepsilon_0,d_0)$-unmated. By Proposition~\ref{SmallDenseBounded} with $k=1$, there exists a subgraph $H$ of $G$ with $\v(H) \le 3d_0$ and $\e(H)\ge \varepsilon_0^2 d_0^2/2$. Hence (i) holds, as desired. So we may assume that $G$ is $(1,\varepsilon_0,d_0)$-unmated. 

Let $F_0$ be an $(\varepsilon_0,d_0)$-mate-free $(\ell+1)$-bounded star forest from $B$ to $A$ such that $\v(F_0)$ is maximized. Note that $B\subseteq V(F_0)$ since $\v(F_0)$ is maximized. Yet $|A\cap V(F_0)| \le \ell |B| < |A|$. Hence $A\setminus V(F_0)\ne \emptyset$.

Choose $u\in A\setminus V(F_0)$. For each $v\in V(G)$ with $v\ne u$, we say that a path $P$ in $G$ from $u$ to $v$ is a \emph{($u,v$)-$F_0$-alternating path} if 
\begin{itemize}
	\item every internal vertex of $P$ has degree exactly one in $F_0 \cap P$ (that is - informally - every other edge of $P$ is in $F_0$), and
	\item there does not exist $u'v'\in E(P)\setminus E(F_0)$ with $u'\in A$, $v'\in B$ and a vertex $w$ in the component of $F_0$ containing $v'$ such that $u'$ and $w$ are $(\varepsilon_0,d_0)$-mates.
\end{itemize}

 Let $F$ be the subgraph of $F_0$ consisting of all the components $T$ of $F_0$ such that there exists a ($u,v$)-$F_0$-alternating path,  where $\{v\}=V(T)\cap B$.

Note that $F$ is non-empty as $u$ has at least $\varepsilon_0 d_0+1$ neighbors in $B$ (since $\varepsilon < 1$) but at most $\varepsilon_0 d_0$ $(\varepsilon_0,d_0)$-mates in $A$ as $G$ is $(1,\varepsilon_0,d_0)$-unmated and ${\rm deg}_G(u) = d_0$. 

\begin{claim}\label{SizeEll}
Every component of $F$ has exactly $\ell$ edges.
\end{claim}
\begin{proof}
Suppose not. That is, there exists a component $T$ of $F$ with $e(T) < \ell$. Let $\{v\}= V(T)\cap B$.	By the definition of $F$, there exists a $(u,v)$-$F_0$-alternating path $P$. Let $F_0' = F_0 \triangle P$. It follows that $F_0'$ is an $(\varepsilon_0,d_0)$-mate-free $(\ell+1)$-bounded star forest from $B$ to $A$. Yet $\v(F_0') > \v(F_0)$, contradicting the choice of $F_0$.
\end{proof}

\begin{claim}\label{NeighborsOut}
Every vertex in $V(F)\cap A$ has at most $\varepsilon_0 d_0$ neighbors in $B\setminus V(F)$.
\end{claim}
\begin{proof}
Suppose not. That is, there exists $w\in V(F)\cap A$ such that $w$ has strictly more than $\varepsilon_0 d_0$ neighbors in $B\setminus V(F)$. Since $G$ is $(1,\varepsilon_0,d_0)$-unmated, it follows that there exists $v\in N(w) \cap B\setminus V(F)$ such that the component of $F_0$ containing $v$ does not contain a $(\varepsilon_0,d_0)$-mate of $w$. 

Let $x\in B$ such that $wx\in E(F)$. By definition of $F$, there exists a $(u,x)$-$F_0$-alternating path $P_0$. If $w\notin V(P_0)$, define $P:= P_0+xw$; otherwise, define $P:= P_0-xw$. Now $P$ is a $(u,w)$-$F_0$-alternating path. But then $P'=P+wv$ is a $(u,v)$-$F_0$-alternating path and hence $v\in V(F)$, a contradiction.
\end{proof}

Let $G_1$ be obtained from $G[V(F)]$ by identifying $A\cap V(C)$ for each component $C$ of $F$. Note that $G_1$ is bipartite. Let $M$ be the perfect matching in $G_1$ corresponding to $F$. Let $G' := G_1/M = G/E(F)$. Then $\v(G_1) = 2\cdot \v(G')$.

By Claim~\ref{SizeEll}, we have that $\v(G[F]) = (\ell+1)\cdot \v(G')$. By Claim~\ref{NeighborsOut} and the fact that every vertex in $A$ has $d_0$ neighbors in $G$, we have that every vertex in $A\cap V(F)$ has degree at least $(1-\varepsilon_0)d_0$ in $G[V(F)]$. Since $F$ is $(\varepsilon_0,d_0)$-mate-free, it follows that every vertex in $V(G_1)\cap A$ has degree at least
$$\ell(1-\varepsilon_0)d_0 - \binom{\ell}{2}\varepsilon_0 d_0 \ge \ell \left(1- \ell \varepsilon_0\right) d_0$$
in $G_1$, where the last inequality follows since $\ell \ge 1$. Since $|V(G_1)\cap A|=|M|=\v(G')$, we have that 
$$\e(G_1) \ge \ell \left(1- \ell \varepsilon_0\right) d_0 \cdot \v(G').$$ 
Since $G_1$ is bipartite, each edge in $G'$ corresponds to at most two edges in $G_1-M$. It follows that
$$\e(G') \ge \frac{\e(G_1)}{2}-|M| \ge \left(\frac{\ell}{2} \left(1- \ell \varepsilon_0\right) d_0-1\right) \cdot \v(G'),$$
and hence 
$$\d(G') \ge \frac{\ell}{2} \left(1-\ell \varepsilon_0\right) d_0 - 1 \geq \frac{\ell}{2} (1-2\ell \varepsilon_0) d_0,$$
where the last inequality follows since $\varepsilon_0 \le \frac{1}{2\ell}$. Since $G'$ is an $(\ell+1)$-bounded minor of $G$, (ii) holds, as desired.
\end{proof}

\subsection{Dense minors in general graphs}\label{Shrub}

The main result of this subsection is Theorem~\ref{DenseSubgraph3}.

First, we need the following definition and proposition. Let $T$ be a tree. We say a vertex $v$ of $T$ is a \emph{centroid} of $T$ if for every edge $e\in E(T)$ incident with $v$, the component of $T-e$ containing $v$ has at least $\v(T)/2$ vertices. Let $v$ be a vertex of $T$ that is not a centroid of $T$. If $e\in E(T)$ is an edge incident with $v$ such that the component $H$ of $T-e$ containing $v$ has at most $\frac{\v(T)-1}{2}$ vertices, then we say $e$ is the \emph{central edge for $v$} in $T$ and that $H$ is the \emph{peripheral piece for $v$}. We need the following theorem of Jordan~\cite{Jordan1869} from 1869 (see~\cite{BLW} for an English translation and history). We include a proof for completeness. 

\begin{proposition}\label{UniqueCenter}
	The number of centroids in a non-empty tree is either one or two.
\end{proposition}
\begin{proof}
Let $T$ be a non-empty tree. If $\v(T)=1$, then $T$ has exactly one centroid as desired. So we may assume that $\v(T)\ge 2$. Now choose $e\in E(T)$ and $T'$ a component of $T-e$ such that $\v(T')\ge \frac{\v(T)}{2}$ and subject to those conditions $\v(T')$ is minimized. Such a choice exists since $\v(T)\ge 2$. If $\v(T')=\frac{\v(T)}{2}$, then the ends of $e$ are precisely the centroids of $T$ as desired. Otherwise $\v(T') > \frac{\v(T)}{2}$ and the end of $e$ in $T'$ is precisely the only centroid of $T$ as desired.
\end{proof}
	
	
	


\begin{restatable}{thm}{DenseThree}\label{DenseSubgraph3}
Let $k \ge \ell \ge 2$ be integers. Let $\varepsilon \in \left(0,\frac{1}{4k}\right)$. Let $G$ be a graph with $d=\d(G) \ge 1/\varepsilon$. Then $G$ contains at least one of the following:
\begin{description}
	\item[(i)] a subgraph $H$ with $\v(H)\le 3k^3d$ and $\e(H)\ge \varepsilon^2d^2/2$, or
	\item[(ii)] a bipartite subgraph $H=(X,Y)$ with $|X| > \ell |Y|$ such that  every vertex in $X$ has at least $(1-2k\varepsilon)d$ neighbors in $Y$, or
	\item[(iii)] a $k$-bounded minor $G'$ with $\d(G') \ge \frac{k}{8\ell} (1-2k\varepsilon) d$.
\end{description}
\end{restatable}

\begin{proof}
Suppose not.	We may assume without loss of generality that  $\d(H) < \d(G)$  for every proper subgraph $H$ of $G$, and hence $\delta(G)>d$. 
	 
Let $A = \{v\in V(G): {\rm deg}(v) \le kd\}$ and $B=V(G)\setminus A$.  Then  $kd|B| \le 2\e(G) = 2d\cdot \v(G)$. Hence  $|B|\le \frac{2}{k}\cdot \v(G)$.  

For a forest $F$ in $G$ define the \emph{$k$-smallness} of $F$ as 
$${\rm small}_k(F) := \sum_{C\in \mathcal{C}(F)} \max\left\{k - 3\cdot \v(C),0\right\},$$
where $\mc{C}(F)$ is the set of components of $F$.	
	
Let $F$ be a $k$-bounded forest with $V(F)=A$ such that
\begin{equation}\label{e:star}
\e(G)-\e(G/E(F)) \le 2\varepsilon d \left( k\cdot \v(G) - {\rm small}_k(F) \right),
\end{equation}
and subject to that ${\rm small}_k(F)$ is minimized. Note that such an $F$ exists as the edgeless graph with $V(F)=A$ is $1$-bounded and satisfies (\ref{e:star}).

Let $G'=G/E(F)$. Since $F$ is $k$-bounded, $G'$ is a $k$-bounded minor of $G$. Thus, since (i) does not hold, we have by Proposition~\ref{SmallDenseBounded} that $G'$ is $(k^2,\varepsilon,d)$-unmated.

Let $C$ be the set of centroids of components $T$ of $F$ with $\v(T) > \frac{2k}{3}$. 	By Proposition~\ref{UniqueCenter}, every component of $F$ has either one or two centroids. Hence $|C| \le 2 \left(\frac{3}{2k}\cdot \v(G)\right) = \frac{3}{k}\cdot \v(G)$. Let $Y=B\cup C$. Then $|Y| \le \frac{5}{k} \cdot \v(G)$. Finally let $X$ denote the set of vertices in components $T$ of $F$ with $\v(T) < \frac{k}{3}$.

\begin{claim}\label{XNeighbors}
Every vertex in $X$ has at least $(1-2k\varepsilon) d$ neighbors in $Y$. 
\end{claim}
\begin{proof}
Suppose not. That is, there exists a vertex $v\in X$ with fewer than $(1-2k\varepsilon) d$ neighbors in $Y$. Since $\delta(G) \ge d$, this implies that $X$ has at least $2k\varepsilon d$ neighbors in $V(G)\setminus Y$. Let $T$ be the component of $F$ containing $v$ and let $x_T$ denote the vertex of $G'$ corresponding to $T$. Since $G'$ is $(k^2,\varepsilon,d)$-unmated and ${\rm deg}_{G'}(x_T)\le \sum_{u \in V(T)}\deg(u) \leq  k^2d$, we have by definition that $x_T$ has at most $\varepsilon d$ $(\varepsilon,d)$-mates in $G'$.

Since $2k\varepsilon d > k\varepsilon d + k-1$, as $d\ge 1/\varepsilon$, it follows that $v$ has a neighbor $u\in V(G)\setminus Y$ in a component $T'$ of $F$ such that $T'\ne T$ and the vertex $x_{T'}$ corresponding to $T'$ in $G'$ is not an $(\varepsilon,d)$-mate of $x_T$ in $G'$. 

First suppose that $\v(T') \le \frac{2k}{3}$. Let $F_1=F+uv$. Thus $T'':=(T \cup T')+uv$ is a component of $F_1$. Since 
$$\v(T'') = \v(T)+\v(T') \le \frac{k}{3} + \frac{2k}{3} \le k.$$
and $F$ is $k$-bounded, we have that $F_1$ is $k$-bounded. Since $x_{T'}$ is not an $(\varepsilon,d)$-mate of $x_T$ in $G'$, we have that 
$$\e(G')-\e(G/E(F_1)) \le \varepsilon d + 1 \le 2\varepsilon d,$$
where the last inequality follows since $d\ge 1/\varepsilon$. Yet 
$${\rm small}_k(F_1) = {\rm small}_k(F) - {\rm small}_k(T) - {\rm small}_k(T') + {\rm small}_k(T'').$$ 
Since ${\rm small}_k(T) > {\rm small}_k(T'')$ and ${\rm small}_k(T') \ge 0$, we have that
$${\rm small}_k(F_1) \le {\rm small}_k(F) - 1,$$
where the $-1$ follows since ${\rm small}_k$ is integral. It now follows that $F_1$ also satisfies (\ref{e:star}). Since $F_1$ is a $k$-bounded forest with $V(F_1)=A$ satisfying (\ref{e:star}) and ${\rm small}_k(F_1) < {\rm small}_k(F)$, we have that $F_1$ contradicts the choice of $F$.

So we may assume that $\v(T')> \frac{2k}{3}$. Since $u\notin Y$, we have by definition that $u$ is not a centroid of $T'$. Let $P$ be the peripheral piece of $T'$ containing $u$ and let $e$ be the central edge of $u$ in $T'$. Since $P$ is a peripheral piece, we have that $\v(P)\le \frac{\v(T')}{2} \le \frac{k}{2}$.

Let $F_2=F+uv-e$. Thus $T_1 := (T\cup P)+uv$ and $T_2 := T'-V(P)$ are components of $F_2$. Now $F_2$ is $k$-bounded since $\v(T_2) < \v(T') \le k$ and
$$\v(T_1) = \v(T)+\v(P) \le \frac{k}{3} + \frac{k}{2} \le k.$$ 
Since $x_{T'}$ is not an $(\varepsilon,d)$-mate of $x_T$ in $G'$, we have that $\e(G')-\e(G/E(F_2)) \le 2\varepsilon d$, as above. Note that $\v(T_1) = \v(T) + \v(P) > \v(T)$ and $\v(T_2) = \v(T) - \v(P) > \frac{\v(T')}{2} \ge \frac{k}{3}$. Yet 
$${\rm small}_k(F_2) = {\rm small}_k(F) - {\rm small}_k(T) - {\rm small}_k(T') + {\rm small}_k(T_1) + {\rm small}_k(T_2).$$ 
Since ${\rm small}_k(T) > {\rm small}_k(T_1)$ and ${\rm small}_k(T')={\rm small}_k(T_2)=0$, we have that
$${\rm small}_k(F_2) \le {\rm small}_k(F) - 1,$$
where the $-1$ follows since ${\rm small}_k$ is integral. But then $F_2$ also satisfies (\ref{e:star}). Since $F_2$ is a $k$-bounded forest with $V(F_2)=A$ satisfying (\ref{e:star}) and ${\rm small}_k(F_2) < {\rm small}_k(F)$, we have that $F_2$ contradicts the choice of $F$.
\end{proof}

We now return to the main proof.  First suppose that $|X| > \ell |Y|$. Let $H$ be the bipartite graph with $V(H)=X\cup Y$ and $E(H) = \{xy\in E(G): x\in X,~y\in Y\}$. Then  (ii) holds by Claim~\ref{XNeighbors}, a contradiction.

So we may assume that $|X| \le \ell |Y| \le \frac{5\ell}{k} \cdot \v(G)$. Note that $F$ has at most $\frac{3}{k} \cdot \v(G)$ components $T$ with $\v(T) \geq \frac{k}{3}$. Thus 
$$\v(G') \le |X| + \frac{3}{k}\cdot \v(G) + |B| \le \frac{5(\ell+1)}{k}\cdot \v(G) \le \frac{8\ell}{k}\cdot \v(G),$$
where the last inequality follows since $\ell \ge 2$. Recall that by construction,
$$\e(G)-\e(G') \le 2 \varepsilon d \left( k\cdot \v(G) - {\rm small}_k(F) \right) \le 2\varepsilon d k \cdot \v(G),$$
where the last inequality follows since ${\rm small}_k(F)\ge 0$. Since $\e(G) =d\cdot \v(G)$, it follows from the inequality above that
$$\e(G') \ge (1-2k\varepsilon)d \cdot \v(G).$$
Since $\v(G')\le \frac{8\ell}{k}\cdot \v(G)$, we have that
$$\d(G') \ge \frac{k}{8\ell} (1-2k\varepsilon) d,$$
and (iii) holds, a contradiction.
\end{proof}

\subsection{Combining the cases}\label{Combined}

We now derive \cref{DenseSubgraph2}, which we restate for convenience, from Theorems~\ref{kclawDense} and~\ref{DenseSubgraph3}. \Dense*

\begin{proof}
	
	We apply Theorem~\ref{DenseSubgraph3} to $G$. If Theorem~\ref{DenseSubgraph3}(i) holds, then (i) holds as desired. Similarly if Theorem~\ref{DenseSubgraph3}(iii) holds, then (iii) holds as desired. 
		
	So we may assume that Theorem~\ref{DenseSubgraph3}(ii) holds, that is there exists a bipartite subgraph $H=(X,Y)$ with $|X| > \ell |Y|$ such that  every vertex in $X$ has at least $(1-2k\varepsilon)d$ neighbors in $Y$.  We next apply Theorem~\ref{kclawDense} with $d_0= (1-2k\varepsilon)d$ and $\varepsilon_0 = 2\varepsilon$ to $H$. 
	
	First assume Theorem~\ref{kclawDense}(i) holds. That is, there exists a subgraph $H_0$ of $H$ with $\v(H_0) \le 3d_0 \le 3k^3d$ and $\e(H_0)\ge \varepsilon_0^2 d_0^2/2 = 4\varepsilon^2 (1-2k\varepsilon)^2 d^2/2.$ Since $2k\varepsilon\le 1/2$ as $\varepsilon \le \frac{1}{4k}$, we find that $\e(H_0)\ge \varepsilon^2 d^2/2$ and (i) holds as desired.
	
	So we may assume that Theorem~\ref{kclawDense}(ii) holds. That is, $H$ contains an $(\ell+1)$-bounded minor $H_0$ with 
	$$\d(H_0) \ge \frac{\ell}{2} (1-2\ell\varepsilon_0) d_0 \ge  \frac{\ell}{2}(1-4k\varepsilon)(1-2k\varepsilon) d \ge \frac{\ell}{2} (1-6k\varepsilon) d,$$ where the middle inequality uses the fact that $\ell \le k$. Hence (ii) holds with $G'=H_0$, as desired.
\end{proof}

\section{An application of \cref{t:newforced}}\label{s:turan}
For a pair of graphs  $G$ and $H$, we say that $G$  is \emph{$H$-free} if no subgraph of $G$ is isomorphic to $H$.
The next theorem due to K\"{u}hn and Osthus~\cite{KO05} shows that $H$-free graphs  have exceptionally dense minors for every complete bipartite graph $H$.  

\begin{thm}[\cite{KO05}]\label{t:KO}
For every integer $s \geq 2$, every $K_{s,s}$-free graph  $G$ has a minor $J$ with  
	\begin{equation}\label{e:KO1}
	\d(J) \geq  (\d(G))^{1 +\frac{1}{2(s-1)}-o_{\d(G)}(1)}.
	\end{equation}
\end{thm}

Krivelevich and Sudakov~\cite{KriSud09} tightened (\ref{e:KO1}) to $\d(J) \geq  c_s(\d(G))^{1 +\frac{1}{s-1}}$ for some $c_s>0$ independent of $\d(G)$. They also proved the following, strengthening a result of K\"{u}hn and Osthus~\cite{KO03}. 

\begin{thm}[\cite{KriSud09}]\label{t:KS}
	For every integer $k \geq 2
	$ there exists $c_k>0$ such that every $C_{2k}$-free $G$ has a minor $J$ with  
	$$
	\d(J) \geq  c_{k}(\d(G))^{\frac{k+1}{2}}.
$$
\end{thm}

The exponents appearing in Theorems \ref{t:KO} and \ref{t:KS} can not be improved, subject to well known conjectures on the Tur\'{a}n numbers of $K_{s,s}$ and $C_{2k}$, which we mention below. 

In this section we use \cref{t:newforced} to extend  Theorems \ref{t:KO} and \ref{t:KS} to general bipartite graphs. Stating our result requires  a couple of definitions. The \emph{Tur\'{a}n number} $\ex(n,H)$ {of a graph $H$} with $\e(H) \neq 0$  is the maximum number of edges in an $H$-free graph $G$ with  $\v(G)=n$. The \emph{Tur\'{a}n exponent} $\gamma(H)$ of a graph $H$ with $\e(H) \geq 2$ is defined as
$$\gamma(H):= \limsup_{n \to \infty} \frac{\log \ex(n,H)}{\log n}.$$
Many fundamental questions about Tur\'{a}n exponents of bipartite graphs remain open. In particular,  a famous conjecture of Erd\H{o}s and Simonovits (see \cite[Conjecture 1.6]{FurSim13}) states that $\gamma(H)$ is rational for every graph $H$, and that $\lim_{n \to \infty} \ex(n,H)/n^{\gamma(H)}$ exists and is positive. We refer the reader to a comprehensive survey by F\"{u}redi and  Simonovits~\cite{FurSim13} for further background. 
 
The main result of this section is an essentially tight analogue of Theorems \ref{t:KO} and \ref{t:KS} for $H$-free graphs $G$ for general bipartite $H$. 

\begin{thm}\label{t:bip}
	For every bipartite graph  $H
	$ with  $\gamma(H) > 1$, every $H$-free  graph $G$ has a minor $J$ with  
	$$
	\d(J) \geq  (\d(G))^{\frac{\gamma(H)}{2(\gamma(H)-1)}-o_{\d(G)}(1)}.
	$$
\end{thm}

\begin{proof}
	The theorem follows from \cref{t:newforced} via a routine, if not exceptionally short, calculation. Let $H$ be as in the theorem statement, and let $\gamma=\gamma(H)$. We need to show that for  every $\eps > 0$, there exists 
	$d_0=d_0(\eps,H)>0$ such that every $H$-free  graph $G$ with $d(G) \geq d_0$ has a minor $J$  with  
	$$
	\d(J) \geq  (\d(G))^{\frac{\gamma}{2(\gamma-1)}-\eps}.
	$$
	Let $\delta $ be chosen so that $\left(\frac{\gamma}{2(\gamma-1)} -\varepsilon \right)\delta \leq \frac{1}{2}$ and \begin{equation}\label{e:delta}\frac{\gamma+\delta}{(2+\delta)(\gamma-1+\delta) + \delta} > \frac{\gamma}{2(\gamma-1)}-\eps,\end{equation} and let $C=C_{\ref{t:newforced}}(\delta)$ be as in \cref{t:newforced}. Let $d_0 \geq C^{1/\delta}$ be chosen so that every $H$-free  graph $G'$ with $v(H) \geq (d_0)^{1/2}$ satisfies  $ \e(G') \leq (\v(G'))^{\gamma+\delta}$. Such a choice is possible by definition of the Tur\'an exponent $\gamma(H)$.
	
	Let $G$ be an $H$-free  graph with $d:=d(G) \geq d_0$, and let  $D:=d^{\frac{\gamma}{2(\gamma-1)}-\eps}$. We assume for a contradiction that the density of every minor of $G$ is less than $D$. Then by \cref{t:newforced} there exists a subgraph $G'$ with \begin{equation}
	\label{e:vgprime}
		\v(G') \leq (D/d)^{\delta} CD^2/d \leq D^{2+\delta}/d,
	\end{equation} and \begin{equation}
	\label{e:dgprime}\d(G') \geq (D/d)^{-\delta}d/C \geq d/D^{\delta}.\end{equation} By the choice of $d_0$, we have that $$\v(G') \geq \d(G') \geq d^{1-\delta\left(\frac{\gamma}{2(\gamma-1)} - \varepsilon\right)} \geq (d_0)^{1/2}.$$
	It follows from the choice of $d_0$ and the fact that $G'$ is $H$-free that $$\d(G')=\frac{\e(G')}{\v(G')} \leq (\v(G'))^{\gamma -1 + \delta}.$$
	Substituting (\ref{e:vgprime}) and (\ref{e:dgprime}) in the above we obtain
	$$
	\frac{d}{D^{\delta}} \leq  \left( \frac{D^{2+\delta}}{d} \right)^{\gamma-1+\delta},  
	$$
	implying
	$$
	D \geq d^{\frac{\gamma+\delta}{(2+\delta)(\gamma-1+\delta) + \delta}} \stackrel{(\ref{e:delta})}{>} d^{\frac{\gamma}{2(\gamma-1)}-\eps} = D,
	$$
	the desired contradiction.
\end{proof}

By definition of the Tur\'an exponent, for every graph $H$  there exists a family of $H$-free graphs $\{G_n\}_{n=1}^{\infty}$ with $\v(G_n) \to \infty$ and $\d(G_n) \geq \v(G_n)^{\gamma(H)-1 -o(1)}$. For every minor $J$ of a graph $G_n$ we have $$\d^2(J) \leq \e(J) \leq \e(G_n) = \v(G_n)\d(G_n) \leq \d(G_n)^{\frac{\gamma(H)}{\gamma(H)-1} +o(1)},$$ assuming $\gamma(H)>1$.
Thus \cref{t:bip} is tight up to the $o_{\d(G)}(1)$ term, as claimed above. 

It has been shown by K\H{o}vari, S\'os and Tur\'an~\cite{KST54} that $\gamma(K_{s,s}) \leq 2- \frac{1}{s}$, and by Erd\H{o}s (see~\cite[Theorem 4.6]{FurSim13}), and Bondy and Simonovits~\cite{BS74} that $\gamma(C_{2k}) \leq (k+1)/k$ . Thus \cref{t:bip} extends Theorems \ref{t:KO} and \ref{t:KS}, although the error term in \cref{t:KS} is better controlled. Note that the tightness of Theorems \ref{t:KO} and \ref{t:KS}, unlike that of \cref{t:bip}, hinges on tightness of the above inequalities on $\gamma(K_{s,s})$  and $\gamma(C_{2k})$, which is widely believed, but is in general open.

\cref{t:bip} does not apply to bipartite graphs  $H
$ with  $\gamma(H)  \leq 1$. However, if $\gamma(H)  \leq 1$ then $H$ is a forest (see~\cite[Corollary 2.28]{FurSim13}). It is not hard to show that   for every forest $H$ with $\v(H) \geq 2$ and every $H$-free graph $G$, we have $\d(G) \leq \v(H)-2$, and so the density of $H$-free graphs is bounded for such $H$. (The exact bound is the subject of the famous Erd\H{o}s-S\'os conjecture~\cite{Erd53}, see also~\cite[Conjecture 6.1]{FurSim13}.) Thus there are no meaningful extensions of asymptotic results such as \cref{t:bip} to this case.

\section{Concluding remarks}\label{s:remarks}

\subsubsection*{Further improvements.}

Further improving the bounds obtained in this paper would require improving or replacing \cref{t:minorfrompieces}, which encapsulates our current procedure for obtaining a $K_t$ minor by linking several smaller pieces.

Answering the following question would help determine the limits of the current approach.

\begin{que}\label{q:1} Does there exist $C > 0$ such that for every integer $t \geq 1$ the following holds?
	
If $G$ is a graph and $H_1,H_2,\ldots,H_r$ are vertex-disjoint subgraphs of $V(G)$ for some $r \geq (\log t)^C$, $\kappa(G) \geq Ct$  and $\kappa(H_i) \geq Ct$ for every $i \in [r]$, then $G$ has a $K_t$ minor.
\end{que}

Note that B\"ohme et al.~\cite{BKMM09} have shown that for every integer $t \geq 1$ there exists $N(t)$ such that every graph $G$ with $\kappa(G) \geq 31(t+1)/2$ and $\v(G) \geq N(t)$  has a $K_t$ minor. Their result implies that if we replace the requirement $r \geq (\log t)^C$ in \cref{q:1} by $r \geq N(t)$, then the modified question has a positive answer.

\subsubsection*{Odd minors.} Given graphs $G$ and $H$ we say that $G$ has \emph{an odd $H$ minor} if a graph isomorphic to $H$ can be obtained from a subgraph $G'$ of $G$ by contracting a set of edges forming a cut in $G'$. Gerards and Seymour (see \cite[p. 115]{JenToft95}) conjectured the following strengthening of Hadwiger's conjecture.

\begin{conj}[Odd Hadwiger's Conjecture]
	\label{odd} 
	For every integer  $t \geq 1$, every graph 
	with no odd  $K_{t}$ minor  is  $(t-1)$-colorable. 
\end{conj}

Geelen, Gerards, Reed, Seymour and Vetta~\cite{GGRSV08}  used \cref{t:KT} to show that every graph with no odd $K_t$ minor is $O(t\sqrt{\log{t}})$-colorable. In~\cite{NorSong19Odd} two of us strengthen \cref{t:main} to show the following.

\begin{thm}[\cite{NorSong19Odd}]\label{t:oddmain} For every $\beta > \frac 1 4 $,
	every graph with no odd $K_t$ minor is $O(t (\log  t)^{\beta})$-colorable.
\end{thm}

The proof of \cref{t:oddmain} follows the same strategy as the proof of \cref{t:main}. \cref{t:newforced} can be used as is, while \cref{c:big} and Theorems~\ref{t:minorfrompieces} and~\ref{t:density} are replaced with more technical versions.

\subsubsection*{List coloring.}

A graph $G$ is said to be \emph{$k$-list colorable} if for every assignment of lists $\{L(v)\}_{v \in V(G)}$ to vertices of $G$ such that $|L(v)| \geq k$ for every $v \in V(G)$, there is a choice of colors $\{c(v)\}_{v \in V(G)}$ such that $c(v) \in L(v)$, and $c(v) \neq c(u)$ for every $uv \in E(G)$. 
Clearly every $k$-list colorable graph is $k$-colorable, but the converse implication does not hold. 
Voigt~\cite{Voigt93} has shown that there exist planar graphs which are not $4$-list colorable. Generalizing the result of~\cite{Voigt93}, Bar\'{a}t, Joret and Wood~\cite{BJW11} constructed graphs with no $K_{3t+2}$ minor which are not $4t$-list colorable for every $t \geq 1$. These results leave open the possibility that the Linear Hadwiger's Conjecture holds for list coloring, as conjectured by Kawarabayashi and Mohar~\cite{KawMoh07}.

\begin{conj}[\cite{KawMoh07}]\label{c:ListHadwiger} There exists $C>0$ such that for every integer $t \geq 1$, every graph with no $K_{t}$ minor is $Ct$-list colorable. 
\end{conj}

In~\cite{NorPos20} two of us extended \cref{t:main} to list coloring. 

\begin{thm}[\cite{NorPos20}]\label{t:list} For every $\beta > \frac 1 4 $,
	every graph with no $K_t$ minor is $O(t (\log  t)^{\beta})$-list-colorable.
\end{thm}

The key new ingredient in the proof of \cref{t:list} is the following bound on the size of sufficiently highly connected graphs with no $K_t$ minor.

\begin{thm}[\cite{NorPos20}]
\label{t:connect} For every  $ \beta > 1/4$ and every integer $t \geq 1$ there exists $C>0$ such that every graph $G$ with $\kappa(G) \geq Ct (\log t)^{\beta}$ and  no $K_t$ minor satisfies $\v(G) \leq t(\log t)^{7/4}$.
\end{thm}

The proof of \cref{t:connect} relies on Theorems \ref{t:newforced} and \ref{t:minorfrompieces} and a new essentially tight bound on the density of unbalanced bipartite graphs with no $K_t$ minor. Note that combining Theorem \ref{t:connect} with Theorems \ref{t:big} and \ref{t:minconnectivity} immediately yields \cref{t:main}. In the proof of \cref{t:connect}  the last two ingredients are replaced by new technical variants, which are applicable to list coloring. 

\subsubsection*{Acknowledgements.} This paper combines the content of two preprints~\cite{NorSong19,Pos19}. In~\cite{NorSong19} the first and third author have shown that 
every graph with no $K_t$ minor is $O(t(\log{t})^{0.354})$-colorable. Subsequently, in~\cite{Pos19} the second author proved  \cref{DenseSubgraph2}, strengthening a similar result used in~\cite{NorSong19}, which yields the current bound. 

The research presented in this paper was in part completed during the visit of the third author to McGill University. Z-X. Song thanks the Department of Mathematics and Statistics, McGill University for its hospitality. 

L. Postle thanks Michelle Delcourt for helpful comments.

\bibliographystyle{alpha}
\bibliography{snorin}

\end{document}